\documentclass[12pt]{article}
\usepackage{amsmath,amsfonts,amsthm,amssymb,mathrsfs,mathtools}
\usepackage{graphicx}
\usepackage[usenames,dvipsnames]{color}
\usepackage{float}
\usepackage{graphicx}
\usepackage{subfigure}
\usepackage{caption}
\usepackage{blindtext}
\usepackage{cite}
\usepackage{url}
\usepackage{caption}
\usepackage{epstopdf}
\usepackage{hyperref}

\usepackage{color}
\usepackage[latin1]{inputenc}

\newcommand{\dint}{\displaystyle\int}

\theoremstyle{plain}
\newtheorem{theorem}{Theorem}[section]

\newtheorem{hy}{Assumption}[section]

\newtheorem{proposition}[theorem]{Proposition}

\theoremstyle{definition}
\newtheorem{definition}[theorem]{Definition}

\theoremstyle{remark}
\newtheorem{remark}[theorem]{Remark}

\numberwithin{equation}{section}
\numberwithin{theorem}{section}

\usepackage{a4wide}
\usepackage{geometry}
\begin{document}
	\renewcommand{\thefootnote}{\fnsymbol{footnote}}
	
	\begin{center}
		{\Large \textbf{The Impact of Pinning Points on Memorylessness in L\'evy Random Bridges}} \\[0pt]
		~\\[0pt] \textbf{Mohammed Louriki} \footnote[1]{Mathematics Department, Faculty of Sciences Semalalia, Cadi Ayyad University, Boulevard Prince Moulay Abdellah,	P. O. Box 2390, Marrakesh 40000, Morocco. 
			E-mail: \texttt{m.louriki@uca.ac.ma}}
		\\[0pt]
	\end{center}
	\begin{abstract}
		Random Bridges have gained significant attention in recent years due to their potential applications in various areas, particularly in information-based asset pricing models. This paper aims to explore the potential influence of the pinning point's distribution on the memorylessness and stochastic dynamics of the bridge process. We introduce L\'evy bridges with random length and random pinning points and analyze their Markov property. Our study demonstrates that the Markov property of L\'evy bridges depends on the nature of the distribution of their pinning points. The law of any random variables can be decomposed into singular continuous, discrete, and absolutely continuous parts with respect to the Lebesgue measure (Lebesgue's decomposition theorem). We show that the Markov property holds when the pinning points' law does not have an absolutely continuous part. Conversely, the L\'evy bridge fails to exhibit Markovian behavior when the pinning point has an absolutely continuous part.
	\end{abstract}
	
	\textbf{Keywords:} L\'evy processes, L\'evy bridges, Markov processes.
	\\ 
	\\
	\textbf{MSC:} 60G40, 60G51, 60G52, 60J25, 60E07, 60E10, 60F99.

	\begin{center}
		\section{Introduction}
		\label{Setion_1}
	\end{center}
	Random Bridge processes have proven to be a successful framework for modeling the information related to future non-defaultable cash flows and the default time in financial models. In this context, the length of the bridge serves as a representation of the default time, while the pinning point models the cash flows. Studies by various researchers have demonstrated the effectiveness and applicability of bridge processes in financial modeling contexts. For instance, to model the flow of information concerning a cash flow occurring at a maturity date $T$, represented by a random variable $Z_T$, Brody et al. \cite{BHM2007} and \cite{BHM2008} used the completed natural filtration generated by the Brownian bridge with length $T$ and pinning point $Z_T$. While, in \cite{BHM}, Brody et al. used random gamma bridges, gamma bridges with deterministic length and random pinning point, to model accumulated losses of large credit portfolios in credit risk management. Moreover, in \cite{HHM} the information-based asset-pricing framework of Brody, Hughston, and Macrina is extended to include a wider class of models for market information. To model the information flow, they introduce a class of processes called L\'evy random bridges, L\'evy bridges with deterministic length and random pinning point, generalising the Brownian bridge and gamma bridge information processes. In \cite{HHM2015} the authors develop a class of non-life reserving models using stable-$1/2$ random bridges to simulate the accumulation of paid claims, allowing for an essentially arbitrary choice of a priori distribution for the ultimate loss. For more research paper in the context of modeling the information through the completed filtration generated by bridges with deterministic length and random pinning point we quote \cite{BL}, \cite{EHL(Brownian-Levy)},and \cite{RuYu}. However, In their work \cite{BBE}, motivated by modelling information concerning credit default times, Bedini, Buckdahn and Engelbert used Brownian bridges with random length, to model the flow of information concerning the default time of a financial company or state. In order to include a wider class of information processes for market filtration Eraoui et al., see \cite{EL}, \cite{EHL}, \cite{EHL(Levy)}, and  introduced and studied Gaussian, Gamma, and L\'evy bridges with random length. This opens the way for different application for information-based approaches for credit risk. In \cite{L}, the Brownian bridge concept is extended by considering uncertainty not only for the time level but also for the pinning point. The paper suggests a Brownian bridge approach for the information flow for the switching behaviour of gas storage contract holders using a two-point distribution for the pinning point. On a different note, in \cite{Lpuqr}, the joint modeling of a non-defaultable cash flow and the insolvency time of the writer of the underlying asset is investigated by introducing a suitable random bridge.
	
	The primary focus of this paper is twofold: firstly, to introduce L\'evy bridges with random length and random pinning point; secondly, to investigate their Markov property. Note that in \cite{L}, the bridge with random length $\tau$ and pinning point $Z$ associated with a Brownian motion is constructed by replacing the deterministic length $r$ and pinning point $z$ with the corresponding values of the random time $\tau$ and random variable $Z$. The availability of an explicit representation of the Brownian bridge is crucial for this construction. However, it should be noted that for a general L\'evy process, the existence of an explicit representation of the bridge is not guaranteed. To introduce a L\'evy bridge with random length $\tau$ and random pinning point $Z$ we choose the following approach: conditionally on the events $\tau=r$ and $Z=z$ the law of the L\'evy bridge with random length $\tau$ and random pinning point $Z$ is none other than that of the L\'evy bridge with length $r$ and pinning point $z$. For this, following the approach given in \cite{FPY} for the construction of Markov bridges with deterministic length $r$ and deterministic pinning point $z$, we will need to assume the existence of transition probability densities of the L\'evy process. Once the construction is completed, our primary objective is to examine the Markov property of the resulting bridge. We show that the pinning point has a crucial impact on the memorylessness and the dynamic of the bridge process. Specifically, from the well-known Lebesgue's decomposition theorem, we see that the law $\mathbb{P}_{Z}$ of $Z$ can be decomposed as
	\begin{equation*}
		\mathbb{P}_{Z}=a_{sd}\mathbb{P}_{Z}^{sd}+a_{sc}\mathbb{P}_{Z}^{sc}+a_{ac}\mathbb{P}_{Z}^{ac}
	\end{equation*}
	where $a_{sd}$, $a_{sc}$, and $a_{ac}$ are three positive real numbers such that $a_{sd}+a_{sc}+a_{ac}=1$, and $\mathbb{P}_{Z}^{sd}$ is discrete (has a countable support), $\mathbb{P}_{Z}^{sc}$ is singular-continuous (has a distribution function that is continuous but not absolutely continuous, it only increases on a set of Lebesgue measure $0$), and  $\mathbb{P}_{Z}^{ac}$ is absolutely continuous (there exists a positive function $f_Z(t)$ such that $\mathbb{P}_{Z}^{ac}(A)=\int_{A} f_Z(s) \mathrm{d} s$ for all $A \in \mathcal{B}(\mathbb{R})$). We demonstrate that if $a_{ac}=0$, then the L\'evy bridge with random length $\tau$ and random pinning point $Z$ is Markovian. However, the presence of the absolutely continuous part in the law of $Z$ destroys the Markov property. Precisely, if $a_{ac}\neq 0$, then there exists times $\tau$ such that the L\'evy bridge with random length $\tau$ and random pinning point $Z$ does not possess the Markov property. Let us delve into the intuitive explanation of how the breakdown of the Markov property in the bridge process with length $\tau$ and pinning point $Z$ is influenced by the absolute continuity of the pinning point $Z$. The random time $\tau$ represents a significant event where the process suddenly exhibits a shift in behavior. Precisely, at time $\tau$, the process transitions to being governed solely by the random variable $Z$. Consequently, the occurrence of $\tau$ marks a crucial point in the process where the Markov property may be tested. The Markov property states that the future behaviour of a stochastic process depends only on its present state, not on the sequence of events that preceded it. Therefore, if the random time $\tau$ can be determined solely based on the present state of the process $\zeta_t$, then the Markov property may be preserved. However, if the occurrence of $\tau$ depends on additional information beyond the present state $\zeta_t$, such as the entire history leading up to $\tau$, then the Markov property may be violated. In this case, the future behavior of the process after $\tau$ becomes dependent on the past history, indicating a breakdown of memorylessness and the Markov property. In the discrete case, where $Z$ has a discrete distribution, determining the occurrence of $\tau$ might be straightforward because the values of $Z$ are distinct and separated, allowing for clear identification of when the process transition $Z$ happens. However, in the case where $Z$ follows an absolutely continuous distribution, the values of $Z$ form a continuum. This smooth variation can make it challenging to identify precisely when the transition occurs based solely on the information available at time $\tau$. This has been confirmed by Proposition \ref{propmeasurablity} in which we have shown that the event $\{\tau\leq t \}$ is measurable with respect to the completed filtration generated by $\zeta_t$ if and only if the law of the pinning point $Z$ does not possess an absolutely continuous part.
	
	The current paper extends the results of the paper \cite{L} in two different ways, firstly in \cite{L} the Markov property is investigated for the Brownian case only for discrete or absolutely continuous pinning point, nothing has been said about the singular continuous case, secondly the proof in \cite{L} relies on the fact that the bridge admits an explicit representation in the Brownian case. This extends also \cite{EHL(Levy)} in which the L\'evy bridge Markov property is studied in the case when the pinning point is deterministic. Moreover, in this paper, different methods have been used. 
	
	In Section 2, we introduce L\'evy bridges with random length and pinning point, we list some examples, and we investigate the bridge Markov property in function of the presence of the absolutely continuous part in the pinning point law. 
	
	The following notations will be used throughout the paper: 
	For a complete probability space $(\Omega,\mathcal{F},\mathbb{P})$, $\mathcal{N}_{\mathbb{P}}$ denotes the
	collection of $\mathbb{P}$-null sets. If $\theta$ is a random variable, then $\mathbb{P}_{\theta}$ denotes the law of $\theta$ under $\mathbb{P}$. If $E$ is a topological space, then the Borel $\sigma$-algebra over $E$ will be denoted by $\mathcal{B}(E)$. The characteristic function of a set $A$ is written $\mathbb{I}_{A}$. 
	The symmetric difference of two sets $A$ and $B$ is denoted by $A\Delta B$. We denote by $(\mathbb{D}_{\infty},\mathcal{T}_{\mathbb{D}_{\infty}})$ is the Skorohod space equipped with the Skorohod topology and $\mathcal{D}$
	denotes its Borel $\sigma$-field. Finally, for any process $Y=(Y_t,\, t\geq 0)$ on $(\Omega,\mathcal{F},\mathbb{P})$, we define by:
	\begin{equation*}
		\mathbb{F}^{Y}=\bigg(\mathcal{F}^{Y}_t:=\sigma(Y_s, s\leq t),~ t\geq 0\bigg) \text{  the natural filtration of the process  }Y.
	\end{equation*}
	\section{L\'evy Bridges with Random Length and Pinning Point}\label{sectionstoppingtimeproperty}
	Motivated by the desire to generalize L\'evy bridges, we introduce the notion of L\'evy bridges with random length and random pinning point. These bridges extend the classical concept of L\'evy bridges by allowing for uncertainty in both the length of the bridge and the location at which it is pinned. This generalization extends remarkably the possibility for the modelling of information flows. In this section, our objective is to define L\'evy bridges with random length and pinning point and investigate their Markov property. A significant finding is that the Markov property of this class of processes depends on the nature of the law of the pinning point of the bridge process. Before introducing this novel class of processes, we will provide a brief review of the definition and properties of L\'evy bridges with deterministic length and pinning point. L\'evy processes are characterized by the triple $(b, \sigma, \nu)$, where $b$ represents the drift component, $\sigma$ describes the covariance structure of the Gaussian part, and $\nu$ denotes the jump intensity or rate at which jumps occur. This triple uniquely determines the behavior and properties of the L\'evy process. Let $X=(X_t, t\geq 0)$ be a one-dimensional L\'evy process, the law of $X_t$ is specified via its characteristic function given by
	$$ \mathbb{E}[\exp(i \lambda X_t)]=\exp(t\psi(\lambda)), \,\,\lambda \in \mathbb{R}, t\geq 0. $$
	The function $\psi$ is commonly referred to as the characteristic exponent of the process $X$. The explicit expression of the characteristic exponent is derived using the well-known L\'evy-Khintchine formula:
	\begin{equation}
		\psi(\lambda)=i\lambda b-\dfrac{\lambda^2 \sigma}{2}+\int_{\mathbb{R}}\left(\exp(i\lambda x)-1-i\lambda x \mathbb{I}_{\{\vert x\vert<1\}}\right)\nu(\mathrm{d}x), \,\,\label{eqLevykinchen}
	\end{equation} 
	where $b \in \mathbb{R},$ $\sigma \in \mathbb{R}_+$ and $\nu$ is a measure concentrated on $\mathbb{R}\backslash\{0\}$, called the L\'evy measure, satisfying $$\int_{\mathbb{R}}(x^2\wedge 1)\nu(\mathrm{d}x) < \infty.$$
	Further details about L\'evy processes can be found in \cite{B} and \cite{Sato}. The Kolmogorov-Daniell theorem allows us to see that the finite-dimensional
	distributions of $X$ induce a probability measure $\mathbb{P}_X$ on the Skorohod space $\mathbb{D}_{\infty}$. Furthermore, it is well-known in the literature that every L\'evy process can be precisely represented as the coordinate process on the Skorohod space $\mathbb{D}_{\infty}$. This realization allows us to treat each L\'evy process as a probability measure on $\mathbb{D}_{\infty}$ and vice versa.
	Our main assumption is the following:
	\begin{hy}\label{hydensityexistence}
		For all $t>0$, the probability law $\mathbb{P}_{X_{t}}(\mathrm{d}x)$ is absolutely continuous with respect to Lebesgue measure, with a density function $f^X_t$.
	\end{hy}
	Numerous studies have been dedicated to investigating the conditions that guarantee the absolute continuity of the probability law $\mathbb{P}_{X_{t}}(\mathrm{d}x)$ with respect to Lebesgue measure. For instance, if the Gaussian component $\sigma$ does not vanish, or if the characteristic function of $X_t$ is integrable for every $t>0$, then the law $\mathbb{P}_{X_{t}}(\mathrm{d}x)$ admits a density function with respect to Lebesgue measure, for every $t>0$. See Sato \cite{Sato}, and Sharpe \cite{Sh}. See also Tucker \cite{T} and Hartman and Wintner \cite{HW} for more conditions that ensure the existence of a density function.\\
Let $r>0$ and $z\in \mathbb{R}$ such that $0<f_r^X(z)<+\infty$, under Assumption \ref{hydensityexistence}, it follows from \cite[Proposition 1]{FPY} that there exists a unique probability measure $\mathbb{P}^{r,z}$ such that
\begin{equation}
	\mathbb{P}^{r,z}\vert_{\mathcal{F}^{X}_{t}}=\dfrac{f^X_{r-t}(z-X_{t})}{f^X_{r}(z)}\,\,\mathbb{P}\vert_{\mathcal{F}^{X}_{t}}\label{eqbridgelawQrz}
\end{equation}
for all $t<r$. Moreover, under $\mathbb{P}^{r,z}$, the process $(X_{t}, 0\leq t<r)$
is a non-homogeneous Markov process with transition densities given
by 
\begin{equation}
	\mathbb{P}^{r,z}\left(X_{t}\in \hbox{d}y\vert\mathcal{F}^{X}_{s}\right)=\mathbb{P}^{r,z}\left(X_{t}\in \mathrm{d}y\vert X_{s}\right)=\dfrac{f^X_{t-s}(y-X_{s})\,f^X_{r-t}(z-y)}{f^X_{r-s}(z-X_{s})}\,\mathrm{d}y,\,s<t<r.\label{eq:f-transition densities}
\end{equation}
Furthermore, $(\mathbb{P}^{r,z})_{z\in \mathbb{R}}$ is a regular version of the family of  conditional probability distributions $\mathbb{P}(.\vert X_r=z)$, $z\in \mathbb{R}$. This implies that $\mathbb{P}^{r,z}$ is the law of the bridge from $0$ to $z$ with deterministic length $r$ associated with the L\'evy process $X$. Let $\left(X^{r,z},t\leq r\right)$ be the process associated with the probability law $\mathbb{P}^{r,z}$. The process constructed $X^{r,z}$ can be realized as the coordinate process on the Skorohod space $\mathbb{D}_{r}$ of c\`adl\`ag functions
from $[0,r]$ to $\mathbb{R}$. Furthermore, it satisfies that $\mathbb{P}(X_{0}^{r,z}=0,X_{r}^{r,z}=z)=1$. Consequently, the finite-dimensional densities
of $X^{r,z}$ exist and, given $x_{0}=t_{0}=0$,
for every $n\in\mathbb{N}$, $0<t_{1}<t_{2}<...<t_{n}<r$, and $(x_{1},x_{2},...,x_{n})\in\mathbb{R}^{n}$,
we have 
\begin{align}
	\mathbb{P}(X_{t_{1}}^{r,z}\in \hbox{d}x_{1},\ldots,X_{t_{n}}^{r,z}\in \hbox{d}x_{n}) & =\dfrac{f^X_{r-t_{n}}(z-x_{n})}{f^X_{r}(z)}\prod_{i=1}^{n}f^X_{t_{i}-t_{i-1}}(x_{i}-x_{i-1})\,\mathrm{d}x_{1}\ldots \mathrm{d}x_{n}.\label{eqfinitedimensionaldensities}
\end{align}
Throughout the subsequent analysis, we extend the process $X^{r,z}$ beyond time $r$ by assigning it the constant value $z$. That is, we identify the process $ \left( X_t^{r,z}, t\geq 0 \right)$ with the process $\tilde{X}^{r,z}:=\left(X_t^{r,z}\mathbb{I}_{\{t<r\}}+z\mathbb{I}_{\{t\geq r\}},\, t\geq 0\right)$. At this point, we are ready to introduce L\'evy bridges with random length and random pinning point associated with the L\'evy process $X$.
\begin{definition}\label{defzeta^z}
	Let $\tau$ and $Z$ be two independent random variables taking values in $(0,+\infty)$ and $\mathbb{R}$, respectively. We say that a process $\zeta$ is the L\'evy bridge with random length $\tau$ and pinning point $Z$ derived from the L\'evy process $X$ if the following are satisfied:
	\begin{enumerate}
		\item[(i)] $0<f^X_r(z)<\infty$ for $\mathbb{P}_{(\tau,Z)}$ almost every $(r,z)$.
		\item[(ii)] The conditional distribution of $\zeta$ given $\left\{ \tau=r, Z=z\right\} $ is the law of the process $X^{r,z}$.
	\end{enumerate}
\end{definition}
\begin{remark}\label{remarkzeta^z}
	Following Definition 3.1 in \cite{EHL(Levy)}, given $Z=z$, $\zeta$ is none other that the L\'evy process pinned in $z$ at the random time $\tau$, or equivalently, the L\'evy bridge with pinning point $z$ and random length $\tau$. Indeed, for all $n\in \mathbb{N}$, $0 < t_1 < t_2 <\ldots < t_n$, and every bounded measurable function defined on $\mathbb{R}^n$, we have
	\begin{align}
		\mathbb{E}[g(\zeta_{t_1},\ldots,\zeta_{t_n})\vert Z=z]&=\displaystyle\int_{0}^{+\infty}\mathbb{E}[g(\zeta_{t_1},\ldots,\zeta_{t_n})\vert Z=z,\tau=r]\mathbb{P}_{\tau\vert Z=z}(\mathrm{d}r)\nonumber\\
		&=\displaystyle\int_{0}^{+\infty}\mathbb{E}[g(X^{r,z}_{t_1},\ldots,X^{r,z}_{t_n})]\mathbb{P}_{\tau}(\mathrm{d}r).
	\end{align}
	Denoting by $\zeta^{\tau,z}$ the bridge with pinning point $z$ and random length $\tau$ associated with $X$, we also have
	\begin{align}
		\mathbb{E}[g(\zeta_{t_1}^{\tau,z},\ldots,\zeta_{t_n}^{\tau,z})]&=\displaystyle\int_{0}^{+\infty}\mathbb{E}[g(\zeta_{t_1}^{\tau,z},\ldots,\zeta_{t_n}^{\tau,z})\vert \tau=r]\mathbb{P}_{\tau}(\mathrm{d}r)\nonumber\\
		&=\displaystyle\int_{0}^{+\infty}\mathbb{E}[g(X^{r,z}_{t_1},\ldots,X^{r,z}_{t_n})]\mathbb{P}_{\tau}(\mathrm{d}r).
	\end{align}	
\end{remark}
\begin{remark}
	The process $\zeta$ can be realized as follows : Consider the probability space $\left(\tilde{\Omega},\tilde{\mathcal{F}},\tilde{\mathbb{P}}\right)$
	\[
	\tilde{\Omega}=\mathbb{D}_{\infty}\times(0,+\infty)\times\mathbb{R},\quad\tilde{\mathcal{F}}=\mathcal{F}\otimes\mathcal{B}\left((0,+\infty)\right)\otimes\mathcal{B}(\mathbb{R})\quad\tilde{\mathbb{P}}\left(\hbox{d}w,\mathrm{d}r,\mathrm{d}z\right)=\mathbb{P}_{X^{r,z}}(\hbox{d}w)\,\mathbb{P}_{(\tau, Z)}(\hbox{d}r,\mathrm{d}z).
	\]
	
	We write $\tilde{w}=(w,r,z)$ for a generic point of $\tilde{\Omega}$.
	Now define 
	\[
	\tilde{\tau}(\tilde{w})=r\text{,}\quad\tilde{Z}(\tilde{w})=z\quad\text{and \quad}\tilde{\zeta}_{t}(\tilde{w})=X_{t}^{r,z}(w),\quad t\geq0.
	\]
	Thus, we have $$ \tilde{\mathbb{P}}\left(\tilde{\tau}\leq t \right) =\mathbb{P}(\tau \leq t),\,\, t\geq 0,$$
	$$ \tilde{\mathbb{P}}\left(\tilde{Z}\leq a \right) =\mathbb{P}(Z \leq a),\,\, a\in \mathbb{R},$$ and the conditional distribution of $\tilde{\zeta}$ given $\left\{ \tilde{\tau}=r, \tilde{Z}=z\right\} $
	is the law $\mathbb{P}_{X^{r,z}}$. Moreover, the process thus defined has c\`adl\`ag paths and  satisfies $\tilde{\zeta}_{t}=\tilde{Z}$ when $\tilde{\tau} \leq t$ and $\tilde{\zeta}_{0}=0$.
\end{remark}
\subsection{Examples}We provide explicit representations of L\'evy bridges with random length and pinning point in the following examples. 
\subsubsection{Brownian motion} Let $(r, z)\in (0, +\infty)\times\mathbb{R}$ and $B$ be a Brownian motion, it is obvious that the Brownian bridge with length $r$ and pining point $z$ can be represented as:
\begin{equation}
	B_t^{r,z}=B_{t\wedge r}-\dfrac{t\wedge r}{r}B_{r}+\dfrac{t\wedge r}{r}\,z,\,\,\,t\geq 0.
\end{equation}			
For any strictly positive random time $\tau$ and any random variable $Z$ such that $\tau$, $Z$, and $B$ are independent, the process given by 
\begin{equation}
	\zeta_t=B_{t\wedge \tau}-\dfrac{t\wedge \tau}{\tau} B_{\tau}+\dfrac{t\wedge \tau}{\tau}\,Z,\,\,\,t\geq 0,
\end{equation}
is the bridge with random length $\tau$ and pinning point $Z$ associated with the Brownian motion $B$. For a fuller treatment we refer the reader to \cite{L}.
\subsubsection{Gamma process}\label{subsectionGamma} Let $\gamma$ be a gamma process, $\tau$ and  $Z$ be two random variables taking values in $(0,+\infty)$ independent of $\gamma$. The process given by
\begin{equation}
	\zeta_t=Z\,\dfrac{\gamma_{t\wedge \tau}}{\gamma_{\tau}},\,\,\,t\geq 0.
\end{equation}
is the bridge with random length $\tau$ and pinning point $Z$. See \cite{EHL} for the case in which the pinning point is deterministic.
\subsubsection{Symmetric $\alpha$-stable process} Let $X^{\alpha}$ be a symmetric $\alpha$-stable processes with $\alpha \in (0,1)\cup(1,2)$, that is, a L{\'e}vy process with L{\'e}vy measure given by
\[
\mathrm{d}\nu^{\alpha}(x)=\left(\frac{1}{\left|x\right|^{1+\alpha}}\mathbb{I}_{\{x<0\}}+
\frac{1}{x^{1+\alpha}}\mathbb{I}_{\{x>0\}}\right)\mathrm{d}x.\]
Its characteristic exponent $\varPsi_{X^{\alpha}}$ has the form
\[
\varPsi_{X^{\alpha}}(u)=-|u|^{\alpha},~ u\in\mathbb{R}.
\]
Hence, it is clear that its characteristic function is integrable. According to \cite{Sh}, it is evident that for all $t>0$ the law of $X^{\alpha}_t$ is absolutely continuous with respect to Lebesgue measure. Furthermore, the density function takes the following form
\begin{equation}
	f^{\alpha}_t(x)= (2\pi)^{-1}\dint_{\mathbb{R}}\exp(ixy)\exp(-t\vert y\vert^{\alpha})\,\mathrm{d}y, \,\, x\in \mathbb{R}.\label{eqfrepresentation}
\end{equation}
It is important to note that for every $x\in \mathbb{R}$ and $t>0$, we have
\begin{equation}
	0<f^{\alpha}_t(x)<+\infty.
\end{equation}
Indeed, if there exists certain $t$ and $x$ such that $f^{\alpha}_t(x)=0$, then there exists a constant $c$ such that $X^{\alpha}_t-ct$ corresponds to either a subordinator or the negation of a subordinator (refer to \cite{Sh}). This contradicts the symmetry of the process $X^{\alpha}$. On the other hand, it is easy to see that
\begin{equation}
	f^{\alpha}_t(x)\leq \pi^{-1}\dint_{0}^{+\infty}\exp(-t y^{\alpha})\,\mathrm{d}y<+\infty.\label{eqfrepresentation}
\end{equation}
Let $(\tau, Z)$ be two random variables taking values in $(0,+\infty)\times\mathbb{R}\setminus\{0\}$ such that $\tau$ and $Z$ are independent of $X^{\alpha}$. We define the following process:
\begin{equation}
	{\displaystyle \zeta_{t}=\begin{cases}
			\dfrac{\tau^{1/\alpha}}{\mathfrak{g}_{Z\,\tau^{1/\alpha}}^{1/\alpha}}X_{t(\mathfrak{g}_{Z\,\tau^{1/\alpha}})/\tau}^{\alpha} & \text{if}\,\,t<\tau\text{,}\\
			Z & \text{if }\,\,t\geq\tau,
	\end{cases}}
\end{equation}
where, for a real-valued random variable $C$, 
\begin{equation}
	{\displaystyle \mathfrak{g}_{C}=\begin{cases}
			0,\quad\quad  \text{if}\,\,\{t\leq1:X_{t-}^{\alpha}=C\,t^{1/\alpha}\}=\emptyset,\\
			\sup\{t\leq1:X_{t-}^{\alpha}=C\,t^{1/\alpha}\}, \text{otherwise, }
	\end{cases}}\label{eq:last passage}
\end{equation}
The process $\zeta$ is well defined. Indeed, we have 
\begin{align*}
	\mathbb{P}\left(\mathfrak{g}_{Z\,\tau^{1/\alpha}}>0\right)&=\displaystyle\int_{0}^{+\infty}\displaystyle\int_{\mathbb{R}\setminus\{0\}}\mathbb{P}\left(\mathfrak{g}_{Z\,\tau^{1/\alpha}}>0\vert\tau=r,Z=z\right)\mathbb{P}_{(\tau,Z)}(\mathrm{d}r,\mathrm{d}z)\\
	&=\displaystyle\int_{0}^{+\infty}\displaystyle\int_{\mathbb{R}\setminus\{0\}}\mathbb{P}\left(\mathfrak{g}_{z\,r^{1/\alpha}}>0\right)\mathbb{P}_{(\tau,Z)}(\mathrm{d}r,\mathrm{d}z)\\
	&=1.
\end{align*}
where the latter equality holds because for any non-zero real number $c$ we have $\mathbb{P}\left(\mathfrak{g}_{c}>0\right)=1$, (see \cite[Theorem 6]{CUB}), and therefore, we obtain that for any $(r,z)\in (0, +\infty)\times\mathbb{R}\setminus\{0\}$, $\mathbb{P}\left(\mathfrak{g}_{z\,r^{1/\alpha}}>0\right)=1$. Using, Theorem 3 in \cite{CUB} and the fact that $\tau$ and $Z$ are independent of $X^{\alpha}$ we show that $\zeta$ is the bridge with length $\tau$ and pinning point $Z$ associated with $X^{\alpha}$.
\subsection{Certain properties of the random time $\tau$}	 
In this subsection, we investigate certain properties of the random time $\tau$ in terms of the nature of the law of the pinning point $Z$.
\begin{proposition}\label{propstop}
	For all $t>0$, we have 
	\begin{equation}\label{stpprop}
		\mathbb{P}\left(\{\zeta_t = Z\} \bigtriangleup \{\tau \leq t\}\right)=0.
	\end{equation}
\end{proposition}
\begin{proof}
	Using the assertion (ii) of Definition \ref{defzeta^z}, the fact that $X_t^{r,z}$ is absolutely continuous with respect to Lebesgue measure for all $r>t>0$, along with the observation that $X_t^{r,z}=z$ when $t\geq r$, we obtain
	\begin{align*}
		&\mathbb{P}\left(\{\zeta_t = Z\} \bigtriangleup \{\tau \leq t\}\right)=\mathbb{P}(\zeta_{t}=Z,t<\tau)+\mathbb{P}(\zeta_{t}\neq Z,\tau\leq t)\\
		&=\dint_{t}^{+\infty}\dint_{\mathbb{R}} \mathbb{P}(\zeta_{t}=Z \vert \tau=r,Z=z) \mathbb{P}_{(\tau,Z)}(\mathrm{d}r,\mathrm{d}z)+\dint_{0}^{t}\dint_{\mathbb{R}} \mathbb{P}(\zeta_{t}\neq Z \vert \tau=r,Z=z)\mathbb{P}_{(\tau,Z)}(\mathrm{d}r,\mathrm{d}z)
		\\  
		&=\dint_{t}^{+\infty}\dint_{\mathbb{R}} \mathbb{P}(X_{t}^{r,z}=z)\mathbb{P}_{(\tau,Z)}(\mathrm{d}r,\mathrm{d}z)+\dint_{0}^{t}\dint_{\mathbb{R}} \mathbb{P}(z\neq z)\mathbb{P}_{(\tau,Z)}(\mathrm{d}r,\mathrm{d}z)=0.
	\end{align*}
	This completes the proof.
\end{proof}
\begin{remark}
	\begin{enumerate}
		\item[(i)] The process $\left(\mathbb{I}_{\{\tau \leq t\}}, t>0\right)$ is a modification, under the probability measure $\mathbb{P}$, of the process $\left(\mathbb{I}_{\left\{\zeta_t=Z\right\}}, t>0\right)$.
		\item[(ii)] The random time $\tau$ is an $\mathbb{F}^{\zeta}\vee \sigma(Z)\vee \mathcal{N}_{\mathbb{P}}$-stopping time.
	\end{enumerate}
\end{remark}
From the well-known Lebesgue's decomposition theorem, we see that the law $\mathbb{P}_{Z}$ of $Z$ can be decomposed as
\begin{equation}\label{eqP_Zdecom1}
	\mathbb{P}_{Z}=(1-a_{ac})\mathbb{P}_{Z}^{s}+a_{ac}\mathbb{P}_{Z}^{ac}
\end{equation}
where $a_{ac}$ is a non-negative real number less than or equal $1$, $\mathbb{P}_{Z}^{ac}$ is absolutely continuous with respect to the Lebesgue measure (there exists a positive function $f_Z(t)$ such that $\mathbb{P}_{Z}^{ac}(A)=\int_{A} f_Z(s) \mathrm{d} s$ for all $A \in \mathcal{B}(\mathbb{R})$), and $\mathbb{P}_{Z}^{s}$ is singular with respect to the Lebesgue measure. Hence, there exists a set $\mathscr{Z}\subset \mathbb{R}$ such that $\mathbb{P}_{Z}^{ac}(\mathscr{Z})=0$ and $\mathbb{P}_{Z}^{s}(\mathscr{Z})=1$.\\
The singular part can be also decomposed into a discrete part $\mathbb{P}_{Z}^{sd}$ (pure point part), and a singular continuous part $\mathbb{P}_{Z}^{sc}$ (has a distribution function that is continuous but not absolutely continuous, it only increases on a set of Lebesgue measure $0$), that is,
\begin{equation}
	\mathbb{P}_{Z}=a_{sd}\mathbb{P}_{Z}^{sd}+a_{sc}\mathbb{P}_{Z}^{sc}+a_{ac}\mathbb{P}_{Z}^{ac}
\end{equation}
where $a_{sd}$, and $a_{sc}$ are two non-negative real numbers such that $a_{sd}+a_{sc}+a_{ac}=1$. We have the following result:
\begin{proposition}\label{propmeasurablity}
	Let $\zeta$ be a L\'evy bridge with length $\tau$ and pinning point $Z$, the event $\{\tau\leq t \}\in \sigma(\zeta_t)\vee \mathcal{N}_{\mathbb{P}}$ for all $t$ if and only if $a_{ac}=0$.
\end{proposition}
\begin{proof}
	Suppose that $a_{ac}=0$, that is, the pinning point $Z$ takes values in a set $\mathscr{Z}$ of Lebesgue measure zero, then the process $\left(\mathbb{I}_{\{\tau \leq t\}}, t>0\right)$ is a modification, under the probability measure $\mathbb{P}$, of the process $\left(\mathbb{I}_{\left\{\zeta_t\in \mathscr{Z}\right\}}, t>0\right)$. Indeed, for all $t>0$, we have
	\begin{align*}
		&\mathbb{P}\left(\{\zeta_t \in \mathscr{Z}\} \bigtriangleup \{\tau \leq t\}\right)=\mathbb{P}(\zeta_{t}\in \mathscr{Z},t<\tau)+\mathbb{P}(\zeta_t\in \mathscr{Z}^c,\tau\leq t)\\
		&=\dint_{t}^{+\infty}\dint_{\mathbb{R}} \mathbb{P}(\zeta_t \in \mathscr{Z} \vert \tau=r,Z=z)\mathbb{P}_{(\tau,Z)}(\mathrm{d}r,\mathrm{d}z)+\dint_{0}^{t}\dint_{\mathbb{R}} \mathbb{P}(\zeta_t\in \mathscr{Z}^c \vert \tau=r,Z=z)\mathbb{P}_{(\tau,Z)}(\mathrm{d}r,\mathrm{d}z)
		\\  
		&=\dint_{t}^{+\infty}\dint_{\mathbb{R}} \mathbb{P}(X_{t}^{r,z}\in \mathscr{Z})\mathbb{P}_{(\tau,Z)}(\mathrm{d}r,\mathrm{d}z)+ \,\mathbb{P}(\tau\leq t,Z\in \mathscr{Z}^c)=0.
	\end{align*}
	In the latter equality, we have made use of the fact that $\mathbb{P}(Z\in \mathscr{Z}^c)=0$, the set $\mathscr{Z}$ is of Lebesgue measure zero, and that the law of $X_{t}^{r,z}$ is absolutely continuous with respect to the Lebesgue measure when $0<t<r$. This leads to the conclusion that for all $t>0$, the event $\{\tau\leq t \}\in \sigma(\zeta_t)\vee \mathcal{N}_{\mathbb{P}}$. 
	\\
	Conversely, assume that $a_{ac}\neq 0$, it suffices to prove that there exist times $t>0$ such that $\mathbb{E}[\mathbb{I}_{\{\tau \leq t\}}|\zeta_t]\neq \mathbb{I}_{\{\tau \leq t\}}$. To establish this it is necessary to determine the law of $\tau$ given $\zeta_t$. It is not difficult to show that $\mathbb{P}_{\zeta_t\vert\tau=r }$, the law of $\zeta_t$ given $\{\tau=r\}$, is absolutely continuous with respect to $\nu(\mathrm{d}x)=\mathbb{P}^{s}_{Z}(\mathrm{d}x)+\mathrm{d}x$. Moreover, its Radon-Nikodym derivative is given by 
	\begin{multline*}
		q_t(r,x)=((1-a_{ac})\mathbb{I}_{\mathscr{Z}}(x)+a_{ac}\mathbb{I}_{\mathscr{Z}^c}(x)f_Z(x))\mathbb{I}_{\{r\leq t\}}+\dint_{\mathbb{R}} \dfrac{f^X_{r-t}(z-x)f^X_t(x)}{f^X_r(z)}\mathbb{P}_{Z}(\mathrm{d}z)\mathbb{I}_{\mathscr{Z}^c}(x)\mathbb{I}_{\{t<r\}}.
	\end{multline*}
	That is, for every $B \in \mathcal{B}(\mathbb{R})$, we have
	$$ \mathbb{P}_{\zeta_t\vert\tau=r }(B)=\dint_{B}q_t(r,x)\nu(\mathrm{d}x).$$
	Thus, applying Bayes formula, it follows that for all bounded measurable functions $h$ defined on $(0,+\infty)$ we have, $\mathbb{P}$-a.s.,
	\begin{align}
		\mathbb{E}[h(\tau)&\vert\zeta_t]=\dfrac{\dint_{0}^{+\infty}h(r)q_t(r,\zeta_t)\mathbb{P}_{\tau}(\mathrm{d}r)}{\dint_{0}^{+\infty}q_t(r,\zeta_t)\mathbb{P}_{\tau}(\mathrm{d}r)}=\dfrac{f_Z(\zeta_t)\dint_{0}^{t}h(r)\mathbb{P}_{\tau}(\mathrm{d}r)}{f_Z(\zeta_t)F_{\tau}(t)}\mathbb{I}_{\mathscr{Z}}(\zeta_t)\nonumber\\
		&+\dfrac{a_{ac}f_Z(\zeta_t)\dint_{0}^{t}h(r)\mathbb{P}_{\tau}(\mathrm{d}r)+f^X_t(\zeta_t)\dint_{t}^{+\infty}h(r)\dint_{\mathbb{R}} \dfrac{f^X_{r-t}(z-\zeta_t)}{f^X_r(z)}\mathbb{P}_{Z}(\mathrm{d}z)\mathbb{P}_{\tau}(\mathrm{d}r)}{a_{ac}f_Z(\zeta_t)F_{\tau}(t)+f^X_t(\zeta_t)\dint_{t}^{+\infty}\dint_{\mathbb{R}} \dfrac{f^X_{r-t}(z-\zeta_t)}{f^X_r(z)}\mathbb{P}_{Z}(\mathrm{d}z)\mathbb{P}_{\tau}(\mathrm{d}r)}\mathbb{I}_{\mathscr{Z}^c}(\zeta_t).\label{eqtaugivenzeta}
	\end{align}
	Consequently, we have, $\mathbb{P}$-a.s.,
	$$ \mathbb{E}[\mathbb{I}_{\{\tau \leq t\}}|\zeta_t]=\mathbb{I}_{\mathscr{Z}}(\zeta_t)+\dfrac{a_{ac}f_Z(\zeta_t)F_{\tau}(t)}{a_{ac}f_Z(\zeta_t)F_{\tau}(t)+f^X_t(\zeta_t)\dint_{t}^{+\infty}\dint_{\mathbb{R}} \dfrac{f^X_{r-t}(z-\zeta_t)}{f^X_r(z)}\mathbb{P}_{Z}(\mathrm{d}z)\mathbb{P}_{\tau}(\mathrm{d}r)}\mathbb{I}_{\mathscr{Z}^c}(\zeta_t).$$
	For $t>0$ such that $0<\mathbb{P}(\tau>t)<1$,
	let us define the random set $\mathfrak{N}_t$ given by
	\begin{equation*}
		\mathfrak{N}_t=\biggl\{ \dint_{t}^{+\infty}\dint_{\mathbb{R}} \dfrac{f^X_{r-t}(z-\zeta_t)f^X_t(\zeta_t)}{f^X_r(z)}\mathbb{P}_{Z}(\mathrm{d}z)\mathbb{P}_{\tau}(\mathrm{d}r)\neq 0,\,\,\zeta_t\in\mathscr{Z}^c,\, f_Z(\zeta_t)\neq 0,\tau\leq t \biggr\}.
	\end{equation*}
	It is easy to see that for $t>0$ such that $0<\mathbb{P}(\tau>t)<1$, we have, $\mathbb{P}$-a.s.,
	$$\dint_{t}^{+\infty}\dint_{\mathbb{R}} \dfrac{f^X_{r-t}(z-\zeta_t)f^X_t(\zeta_t)}{f^X_r(z)}\mathbb{P}_{Z}(\mathrm{d}z)\mathbb{P}_{\tau}(\mathrm{d}r)\neq 0.$$ 
	Thus, it follows from Proposition \ref{propstop}, the independence of $\tau$ and $Z$, and the decomposition \eqref{eqP_Zdecom1} that
	\begin{equation}
		\mathbb{P}(\mathfrak{N}_t)=\mathbb{P}(\zeta_t\in\mathscr{Z}^c,\, f_Z(\zeta_t)\neq 0,\tau\leq t)=F_{\tau}(t)\mathbb{P}(Z\in\mathscr{Z}^c,\, f_Z(Z)\neq 0)=a_{ac}F_{\tau}(t)>0.
	\end{equation}
	Hence, for $\omega\in \mathfrak{N}_t$, we have 
	\begin{align}
		\dfrac{a_{ac}f_Z(\zeta_t(\omega))F_{\tau}(t)}{a_{ac}f_Z(\zeta_t(\omega))F_{\tau}(t)+f^X_t(\zeta_t(\omega))\dint_{t}^{+\infty}\dint_{\mathbb{R}} \dfrac{f^X_{r-t}(z-\zeta_t(\omega))}{f^X_r(z)}\mathbb{P}_{Z}(\mathrm{d}z)\mathbb{P}_{\tau}(\mathrm{d}r)}<1,
	\end{align}
	this yields that for $\omega\in \mathfrak{N}_t$, we have 
	\begin{align}
		\mathbb{E}[\mathbb{I}_{\{\tau \leq t\}}|\zeta_t](\omega)<\mathbb{I}_{\{\tau \leq t\}}(\omega).
	\end{align}
	Then, we conclude that there exist certain $t>0$, such that, $\mathbb{P}$-a.s.,
	$ \mathbb{E}[\mathbb{I}_{\{\tau \leq t\}}\vert\zeta_t]\neq \mathbb{I}_{\{\tau \leq t\}}$,
	which completes the proof.
\end{proof}		
\subsection{Markov property of the L\'evy bridge $\zeta$}
In this subsection, we investigate the influence of the distribution of the pinning point $Z$ on the memorylessness of the bridge process $\zeta$. The following theorem states that if the law $\mathbb{P}_{Z}$ of the pinning point $Z$ does not possess an absolutely continuous part, then the bridge process satisfies the Markov property.
\begin{theorem}\label{thmMarkov}
	Assume that $a_{ac}=0$, hence, the bridge $\zeta$ with random length $\tau$ and pinning point $Z$ associated with the L\'evy process $X$ is Markovian.
\end{theorem}
\begin{proof}
	Since $\zeta_0=0$, it suffices to prove that for all $n\in \mathbb{N}$, for every bounded measurable function $g$ defined on $\mathbb{R}$, and 
	$0 < t_1 < t_2 <\ldots < t_n < u$, we have $\mathbb{P}$-a.s.,
	\begin{equation}
		\mathbb{E}[g(\zeta_u)\vert\zeta_{t_1},\ldots,\zeta_{t_n}]=\mathbb{E}[g(\zeta_u)\vert\zeta_{t_n}].
	\end{equation}
	Since $a_{ac}=0$, it follows from Proposition \ref{propmeasurablity} that
	\begin{equation*}
		\mathbb{E}[g(\zeta_u)\mathbb{I}_{\{\tau\leq t_n\}}\vert\zeta_{t_1},\ldots,\zeta_{t_n}]=\mathbb{E}[g(\zeta_{t_n})\mathbb{I}_{\{\tau\leq t_n\}}\vert\zeta_{t_1},\ldots,\zeta_{t_n}]=g(\zeta_{t_n})\mathbb{I}_{\{\tau\leq t_n\}}=\mathbb{E}[g(\zeta_u)\mathbb{I}_{\{\tau\leq t_n\}}\vert\zeta_{t_n}].
	\end{equation*}
	Hence, it remains to show that
	\begin{equation*}
		\mathbb{E}[g(\zeta_u)\mathbb{I}_{\{t_n<\tau\}}\vert\zeta_{t_1},\ldots,\zeta_{t_n}]=\mathbb{E}[g(\zeta_u)\mathbb{I}_{\{t_n<\tau\}}\vert\zeta_{t_n}].
	\end{equation*}
	Thus, it suffices to verify that for every bounded measurable function $L$ on $\mathbb{R}^n$, we have
	\begin{equation}
		\mathbb{E}[g(\zeta_u)\mathbb{I}_{\{t_n<\tau\}}L(\zeta_{t_1},\ldots,\zeta_{t_n})]=\mathbb{E}[\mathbb{E}[g(\zeta_u)\vert\zeta_{t_n}]\mathbb{I}_{\{t_n<\tau\}}L(\zeta_{t_1},\ldots,\zeta_{t_n})].\label{eqMart<tau}
	\end{equation}
	To show \eqref{eqMart<tau}, we first need to compute the law of $\zeta_u$ given $\zeta_{t}$ for $0<t<u$. From \eqref{eqtaugivenzeta}, we have for any bounded measurable function $h$ defined on $(0,+\infty)$ we have, $\mathbb{P}$-a.s.,
	\begin{align*}
		\mathbb{E}[h(\tau)|\zeta_t]&=\dint_{0}^{t}\dfrac{h(r)}{F_{\tau}(t)}\mathbb{P}_{\tau}(\mathrm{d}r)\mathbb{I}_{\mathscr{Z}}(\zeta_t)+\dfrac{\dint_{t}^{+\infty}h(r)\dint_{\mathbb{R}} \dfrac{f^X_{r-t}(z-\zeta_t)}{f^X_r(z)}\mathbb{P}_Z(\mathrm{d}z)\mathbb{P}_{\tau}(\mathrm{d}r)}{\dint_{t}^{+\infty}\dint_{\mathbb{R}} \dfrac{f^X_{r-t}(z-\zeta_t)}{f^X_r(z)}\mathbb{P}_Z(\mathrm{d}z)\mathbb{P}_{\tau}(\mathrm{d}r)}\mathbb{I}_{\mathscr{Z}^c}(\zeta_t).
	\end{align*}
	Hence, for any bounded measurable function $U$ defined on $(0,+\infty)\times \mathbb{R}$, with the notation $\zeta^{r,Z}$ is the bridge with random pinning point $Z$ and deterministic length $r$ associated with $X$, we have, $\mathbb{P}$-a.s.,
	\begin{align}
		\mathbb{E}[U(\tau,&Z)\vert\zeta_t=x]=\displaystyle\int_{0}^{+\infty}\mathbb{E}[U(\tau,Z)\vert\zeta_t=x,\tau=r]\mathbb{P}_{\tau\vert \zeta_t=x}(\mathrm{d}r)=\displaystyle\int_{0}^{+\infty}\mathbb{E}[U(r,Z)\vert\zeta^{r,Z}_t=x]\mathbb{P}_{\tau\vert \zeta_t=x}(\mathrm{d}r)\nonumber\\
		&=\dint_{0}^{t}\dfrac{U(r,x)}{F_{\tau}(t)}\mathbb{P}_{\tau}(\mathrm{d}r)\mathbb{I}_{\mathscr{Z}}(x)+\dfrac{\dint_{\mathbb{R}}\dint_{t}^{+\infty}U(r,z) \dfrac{f^X_{r-t}(z-x)}{f^X_r(z)}\mathbb{P}_{\tau}(\mathrm{d}r)\mathbb{P}_Z(\mathrm{d}z)}{\dint_{\mathbb{R}}\dint_{t}^{+\infty} \dfrac{f^X_{r-t}(z-x)}{f^X_r(z)}\mathbb{P}_{\tau}(\mathrm{d}r)\mathbb{P}_Z(\mathrm{d}z)}\mathbb{I}_{\mathscr{Z}^c}(x),
	\end{align}
	where we have used in the latter equality the fact that for any bounded measurable function $g$ defined on $\mathbb{R}$ we have, $\mathbb{P}$-a.s.,
	\begin{equation}
		\mathbb{E}\left[g(Z) \vert \zeta_t^{r, Z}=x\right]=g(x) \mathbb{I}_{\{r \leq t\}}+\frac{\displaystyle\int_{\mathbb{R}} g(z) \dfrac{f^X_{r-t}(z-x)}{f^X_r(z)}\mathbb{P}_{Z}(\mathrm{d}z)}{\displaystyle\int_{\mathbb{R}} \dfrac{f^X_{r-t}(z-x)}{f^X_r(z)}\mathbb{P}_{Z}(\mathrm{d}z)} \mathbb{I}_{\{t<r\}},\label{eqV(r,Z)}
	\end{equation}
	which is obtained by a simple application of Bayes' theorem. By utilizing the fact that
	\begin{equation*}
		\mathbb{E}[g(\zeta_u)\vert\zeta_t=x]=\displaystyle\int_{0}^{+\infty}\displaystyle\int_{\mathbb{R}}\mathbb{E}[g(X_u^{r,z})\vert X^{r,z}_t=x]\mathbb{P}_{(\tau,Z)\vert \zeta_t=x}(\mathrm{d}r,\mathrm{d}z),
	\end{equation*}
	we obtain
	\begin{align}
		\mathbb{E}[g(\zeta_u)\vert\zeta_t]&=g(\zeta_t)\mathbb{I}_{\{\tau\leq t\}}+\dfrac{\dint_{\mathbb{R}}g(z)\dint_{t}^{u} \dfrac{f^X_{r-t}(z-\zeta_t)}{f^X_r(z)}\mathbb{P}_{\tau}(\mathrm{d}r)\mathbb{P}_Z(\mathrm{d}z)}{\dint_{\mathbb{R}}\dint_{t}^{+\infty} \dfrac{f^X_{r-t}(z-\zeta_t)}{f^X_r(z)}\mathbb{P}_{\tau}(\mathrm{d}r)\mathbb{P}_Z(\mathrm{d}z)}\mathbb{I}_{\{t<\tau\}}\nonumber\\
		&+\displaystyle\int_{\mathbb{R}}g(y)f_{u-t}^X(y-\zeta_t)\dfrac{\dint_{\mathbb{R}}\dint_{u}^{+\infty} \dfrac{f^X_{r-u}(z-y)}{f^X_r(z)}\mathbb{P}_{\tau}(\mathrm{d}r)\mathbb{P}_Z(\mathrm{d}z)}{\dint_{\mathbb{R}}\dint_{t}^{+\infty} \dfrac{f^X_{r-t}(z-\zeta_t)}{f^X_r(z)}\mathbb{P}_{\tau}(\mathrm{d}r)\mathbb{P}_Z(\mathrm{d}z)}\mathrm{d}y\,\mathbb{I}_{\{t<\tau\}}.\label{eqlawofzeta_ugivenzeta_t}
	\end{align}
	We are now in a position to demonstrate \eqref{eqMart<tau}. We have,
\begin{small}
	\begin{multline*}
		\mathbb{E}[\mathbb{E}[g(\zeta_u)\vert\zeta_{t_n}]\mathbb{I}_{\{t_n<\tau\}}L(\zeta_{t_1},\ldots,\zeta_{t_n})]=\mathbb{E}\Bigg[\dfrac{\dint_{\mathbb{R}}g(z)\dint_{t_n}^{u} \dfrac{f^X_{r-t_n}(z-\zeta_{t_n})}{f^X_r(z)}\mathbb{P}_{\tau}(\mathrm{d}r)\mathbb{P}_Z(\mathrm{d}z)}{\dint_{\mathbb{R}}\dint_{t_n}^{+\infty} \dfrac{f^X_{r-t_n}(z-\zeta_{t_n})}{f^X_r(z)}\mathbb{P}_{\tau}(\mathrm{d}r)\mathbb{P}_Z(\mathrm{d}z)}\mathbb{I}_{\{t_n<\tau\}}L(\zeta_{t_1},\ldots,\zeta_{t_n})\Bigg]\\
		+\mathbb{E}\Bigg[\displaystyle\int_{\mathbb{R}}g(y)f_{u-t_n}^X(y-\zeta_{t_n})\dfrac{\dint_{\mathbb{R}}\dint_{u}^{+\infty} \dfrac{f^X_{r-u}(z-y)}{f^X_r(z)}\mathbb{P}_{\tau}(\mathrm{d}r)\mathbb{P}_Z(\mathrm{d}z)}{\dint_{\mathbb{R}}\dint_{t_n}^{+\infty} \dfrac{f^X_{r-t_n}(z-\zeta_{t_n})}{f^X_r(z)}\mathbb{P}_{\tau}(\mathrm{d}r)\mathbb{P}_Z(\mathrm{d}z)}\mathrm{d}y\,\mathbb{I}_{\{t_n<\tau\}}L(\zeta_{t_1},\ldots,\zeta_{t_n})\Bigg].\label{eqsplit2Markov}
	\end{multline*}
\end{small}
	From Definition \ref{defzeta^z} and the formula of total probability, we have
	\begin{multline*}
		\mathbb{E}\Bigg[\dfrac{\dint_{\mathbb{R}}g(z)\dint_{t_n}^{u} \dfrac{f^X_{r-t_n}(z-\zeta_{t_n})}{f^X_r(z)}\mathbb{P}_{\tau}(\mathrm{d}r)\mathbb{P}_Z(\mathrm{d}z)}{\dint_{\mathbb{R}}\dint_{t_n}^{+\infty} \dfrac{f^X_{r-t_n}(z-\zeta_{t_n})}{f^X_r(z)}\mathbb{P}_{\tau}(\mathrm{d}r)\mathbb{P}_Z(\mathrm{d}z)}\mathbb{I}_{\{t_n<\tau\}}L(\zeta_{t_1},\ldots,\zeta_{t_n})\Bigg]=\\
		\dint_{\mathbb{R}}\dint_{t_n}^{+\infty}\mathbb{E}\Bigg[\dfrac{\dint_{\mathbb{R}}g(z)\dint_{t_n}^{u} \dfrac{f^X_{r-t_n}(z-X^{r',z'}_{t_n})}{f^X_r(z)}\mathbb{P}_{\tau}(\mathrm{d}r)\mathbb{P}_Z(\mathrm{d}z)}{\dint_{\mathbb{R}}\dint_{t_n}^{+\infty} \dfrac{f^X_{r-t_n}(z-X^{r',z'}_{t_n})}{f^X_r(z)}\mathbb{P}_{\tau}(\mathrm{d}r)\mathbb{P}_Z(\mathrm{d}z)}L(X^{r',z'}_{t_1},\ldots,X^{r',z'}_{t_n})\Bigg]\mathbb{P}_Z(\mathrm{d}z')\mathbb{P}_{\tau}(\mathrm{d}r'),
	\end{multline*}
	Using the fact that the bridge law $\mathbb{P}^{r,z}$ is absolutely continuous with respect to $\mathbb{P}$ with density $M^{r,z}$ given by
	\begin{equation}
		M_t^{r,z}=\dfrac{\mathrm{d}\mathbb{P}^{r,z}\vert_{\mathcal{F}^{X}_{t}}}{\mathrm{d}\mathbb{P}\vert_{\mathcal{F}^{X}_{t}}}=\dfrac{f^X_{r-t}(z-X_{t})}{f^X_{r}(z)},\label{eqbridgelawQrz}
	\end{equation}
	that is, for any bounded $\mathcal{F}^{X}_{t}$- measurable $F$ we have
	\begin{equation}
		\mathbb{E}^{r,z}[F]=\mathbb{E}\bigg[F\,\dfrac{f^X_{r-t}(z-X_{t})}{f^X_{r}(z)}\bigg],\label{eqF}
	\end{equation}
	where $\mathbb{E}^{r,z}$ is the expectation under $\mathbb{P}^{r,z}$, we have 
	\begin{multline*}
		\mathbb{E}\Bigg[\dfrac{\dint_{\mathbb{R}}g(z)\dint_{t_n}^{u} \dfrac{f^X_{r-t_n}(z-X^{r',z'}_{t_n})}{f^X_r(z)}\mathbb{P}_{\tau}(\mathrm{d}r)\mathbb{P}_Z(\mathrm{d}z)}{\dint_{\mathbb{R}}\dint_{t_n}^{+\infty} \dfrac{f^X_{r-t_n}(z-X^{r',z'}_{t_n})}{f^X_r(z)}\mathbb{P}_{\tau}(\mathrm{d}r)\mathbb{P}_Z(\mathrm{d}z)}L(X^{r',z'}_{t_1},\ldots,X^{r',z'}_{t_n})\Bigg]=\\
		\mathbb{E}\Bigg[\dfrac{\dint_{\mathbb{R}}g(z)\dint_{t_n}^{u} \dfrac{f^X_{r-t_n}(z-X_{t_n})}{f^X_r(z)}\mathbb{P}_{\tau}(\mathrm{d}r)\mathbb{P}_Z(\mathrm{d}z)}{\dint_{\mathbb{R}}\dint_{t_n}^{+\infty} \dfrac{f^X_{r-t_n}(z-X_{t_n})}{f^X_r(z)}\mathbb{P}_{\tau}(\mathrm{d}r)\mathbb{P}_Z(\mathrm{d}z)}L(X_{t_1},\ldots,X_{t_n})\dfrac{f^X_{r'-t_n}(z'-X_{t_n})}{f^X_{r'}(z')}\Bigg],
	\end{multline*}
	Fubini and \eqref{eqF} yield,
\begin{align*}
	\mathbb{E}\Bigg[&\dfrac{\dint_{\mathbb{R}}g(z)\dint_{t_n}^{u} \dfrac{f^X_{r-t_n}(z-\zeta_{t_n})}{f^X_r(z)}\mathbb{P}_{\tau}(\mathrm{d}r)\mathbb{P}_Z(\mathrm{d}z)}{\dint_{\mathbb{R}}\dint_{t_n}^{+\infty} \dfrac{f^X_{r-t_n}(z-\zeta_{t_n})}{f^X_r(z)}\mathbb{P}_{\tau}(\mathrm{d}r)\mathbb{P}_Z(\mathrm{d}z)}\mathbb{I}_{\{t_n<\tau\}}L(\zeta_{t_1},\ldots,\zeta_{t_n})\Bigg]\\
	&\quad\quad\quad\quad\quad\quad=\mathbb{E}\Bigg[\dint_{\mathbb{R}}g(z)\dint_{t_n}^{u} \dfrac{f^X_{r-t_n}(z-X_{t_n})}{f^X_r(z)}\mathbb{P}_{\tau}(\mathrm{d}r)\mathbb{P}_Z(\mathrm{d}z)L(X_{t_1},\ldots,X_{t_n})\Bigg]\\
	&\quad\quad\quad\quad\quad\quad=\dint_{\mathbb{R}}g(z)\dint_{t_n}^{u}\mathbb{E}\Bigg[L(X_{t_1},\ldots,X_{t_n})\dfrac{f^X_{r-t_n}(z-X_{t_n})}{f^X_r(z)}\Bigg]\mathbb{P}_{\tau}(\mathrm{d}r)\mathbb{P}_Z(\mathrm{d}z)\\
	&\quad\quad\quad\quad\quad\quad=\dint_{\mathbb{R}}g(z)\dint_{t_n}^{u}\mathbb{E}[L(X^{r,z}_{t_1},\ldots,X^{r,z}_{t_n})]\mathbb{P}_{\tau}(\mathrm{d}r)\mathbb{P}_Z(\mathrm{d}z)\\
	&\quad\quad\quad\quad\quad\quad=\mathbb{E}[g(\zeta_{u})\mathbb{I}_{\{t_n<\tau\leq u\}}L(\zeta_{t_1},\ldots,\zeta_{t_n})].
\end{align*}
	Similarly, Definition \ref{defzeta^z}, Fubini and \eqref{eqF} yield that
	\begin{align*}
		\mathbb{E}\Bigg[\displaystyle\int_{\mathbb{R}}&g(y)f_{u-t_n}^X(y-\zeta_{t_n})\dfrac{\dint_{\mathbb{R}}\dint_{u}^{+\infty} \dfrac{f^X_{r-u}(z-y)}{f^X_r(z)}\mathbb{P}_{\tau}(\mathrm{d}r)\mathbb{P}_Z(\mathrm{d}z)}{\dint_{\mathbb{R}}\dint_{t_n}^{+\infty} \dfrac{f^X_{r-t_n}(z-\zeta_{t_n})}{f^X_r(z)}\mathbb{P}_{\tau}(\mathrm{d}r)\mathbb{P}_Z(\mathrm{d}z)}\mathrm{d}y\,\mathbb{I}_{\{t_n<\tau\}}L(\zeta_{t_1},\ldots,\zeta_{t_n})\Bigg]\\
		&=\mathbb{E}\Bigg[\displaystyle\int_{\mathbb{R}}g(y)f_{u-t_n}^X(y-X_{t_n})\dint_{\mathbb{R}}\dint_{u}^{+\infty} \dfrac{f^X_{r-u}(z-y)}{f^X_r(z)}\mathbb{P}_{\tau}(\mathrm{d}r)\mathbb{P}_Z(\mathrm{d}z)\mathrm{d}yL(X_{t_1},\ldots,X_{t_n})\Bigg]\\
		&=\dint_{\mathbb{R}}\dint_{u}^{+\infty}\mathbb{E}\Bigg[\displaystyle\int_{\mathbb{R}}g(y) \dfrac{f_{u-t_n}^X(y-X_{t_n})f^X_{r-u}(z-y)}{f^X_r(z)}\mathrm{d}yL(X_{t_1},\ldots,X_{t_n})\Bigg]\mathbb{P}_{\tau}(\mathrm{d}r)\mathbb{P}_Z(\mathrm{d}z)\\
		&=\dint_{\mathbb{R}}\dint_{u}^{+\infty}\mathbb{E}\Bigg[\displaystyle\int_{\mathbb{R}}g(y) \dfrac{f_{u-t_n}^X(y-X^{r,z}_{t_n})f^X_{r-u}(z-y)}{f_{r-t_n}^X(z-X^{r,z}_{t_n})}\mathrm{d}yL(X^{r,z}_{t_1},\ldots,X^{r,z}_{t_n})\Bigg]\mathbb{P}_{\tau}(\mathrm{d}r)\mathbb{P}_Z(\mathrm{d}z).
	\end{align*}
	On the other hand, using the Markov property of $X^{r,z}$ we obtain
	\begin{align*}
		\dint_{\mathbb{R}}\dint_{u}^{+\infty}\mathbb{E}\Bigg[\displaystyle\int_{\mathbb{R}}g(y)& \dfrac{f_{u-t_n}^X(y-X^{r,z}_{t_n})f^X_{r-u}(z-y)}{f_{r-t_n}^X(z-X^{r,z}_{t_n})}\mathrm{d}yL(X^{r,z}_{t_1},\ldots,X^{r,z}_{t_n})\Bigg]\mathbb{P}_{\tau}(\mathrm{d}r)\mathbb{P}_Z(\mathrm{d}z)\\
		&=\dint_{\mathbb{R}}\dint_{u}^{+\infty}\mathbb{E}[\mathbb{E}[g(X^{r,z}_{u})\vert X^{r,z}_{t_n}]L(X^{r,z}_{t_1},\ldots,X^{r,z}_{t_n})]\mathbb{P}_{\tau}(\mathrm{d}r)\mathbb{P}_Z(\mathrm{d}z)\\
		&=\dint_{\mathbb{R}}\dint_{u}^{+\infty}\mathbb{E}[g(X^{r,z}_{u})L(X^{r,z}_{t_1},\ldots,X^{r,z}_{t_n})]\mathbb{P}_{\tau}(\mathrm{d}r)\mathbb{P}_Z(\mathrm{d}z)\\
		&=\mathbb{E}[g(\zeta_u)\mathbb{I}_{\{u<\tau\}}L(\zeta_{t_1},\ldots,\zeta_{t_n})].
	\end{align*}
	Combining all this we obtain the asserted result \eqref{eqMart<tau}. The proof of the theorem is completed.
\end{proof}
We are now ready to state the second main result of this paper. This result highlights the importance of the pinning point's distribution in determining the memorylessness and stochastic dynamics of the bridge process.
\begin{theorem}
	Assume that $a_{ac}\neq 0$, hence, there exists non-constant random times $\tau$ such that the L\'evy bridge $\zeta$, with length $\tau$ and pinning point $Z$, does not possess the Markov property with respect to its natural filtration.
\end{theorem}
\begin{proof}
	Assume that $\zeta$ is Markovian. Hence, for all $0<t<u$, and for all bounded continuous function $g$ we have, $\mathbb{P}$-a.s.,
	\begin{equation}
		\mathbb{E}[g(\zeta_{u}) \vert \mathcal{F}^{\zeta}_{t}]=\mathbb{E}[g(\zeta_{u}) \vert \zeta_{t}].
	\end{equation}
	Since $\zeta_u=Z$, on $\{ \tau\leq u\}$ and $\mathbb{P}(\tau<+\infty)=1$, $\zeta_u$ converges to $Z$ a.s. as $u$ goes to $+\infty$. Thus,
	\begin{equation}
		\mathbb{E}[g(Z) \vert \mathcal{F}^{\zeta}_{t}]=\lim\limits_{u\rightarrow +\infty}\mathbb{E}[g(\zeta_{u}) \vert \mathcal{F}^{\zeta}_{t}]=\lim\limits_{u\rightarrow +\infty}\mathbb{E}[g(\zeta_{u}) \vert \zeta_{t}]=\mathbb{E}[g(Z) \vert \zeta_{t}].\label{eqAssumeMarkov}
	\end{equation}
	With the notation $\zeta^{r,Z}$ is the bridge with random pinning point $Z$ and deterministic length $r$ associated with $X$, we have, $\mathbb{P}$-a.s.,
	\begin{align}
		\mathbb{E}[g(Z)\vert\zeta_t=x]&=\displaystyle\int_{0}^{+\infty}\mathbb{E}[g(Z)\vert\zeta_t=x,\tau=r]\mathbb{P}_{\tau\vert \zeta_t=x}(\mathrm{d}r)=\displaystyle\int_{0}^{+\infty}\mathbb{E}[g(Z)\vert\zeta^{r,Z}_t=x]\mathbb{P}_{\tau\vert \zeta_t=x}(\mathrm{d}r).\label{eqV(tau,Z)}
	\end{align}
	Thus, \eqref{eqtaugivenzeta}, \eqref{eqV(r,Z)}, and \eqref{eqV(tau,Z)} show that
	\begin{align}
		\mathbb{E}[g(Z)\vert\zeta_t]&=g(\zeta_t)\mathbb{I}_{\mathscr{Z}}(\zeta_t)+\nonumber\\&\dfrac{a_{ac}f_Z(\zeta_t)g(\zeta_t)F_{\tau}(t)+f^X_t(\zeta_t)\dint_{t}^{+\infty}\dint_{\mathbb{R}}g(z) \dfrac{f^X_{r-t}(z-\zeta_t)}{f^X_r(z)}\mathbb{P}_{Z}(\mathrm{d}z)\mathbb{P}_{\tau}(\mathrm{d}r)}{a_{ac}f_Z(\zeta_t)F_{\tau}(t)+f^X_t(\zeta_t)\dint_{t}^{+\infty}\dint_{\mathbb{R}} \dfrac{f^X_{r-t}(z-\zeta_t)}{f^X_r(z)}\mathbb{P}_{Z}(\mathrm{d}z)\mathbb{P}_{\tau}(\mathrm{d}r)}\mathbb{I}_{\mathscr{Z}^c}(\zeta_t).\label{erZgivenzeta}
	\end{align}
	Let $0<t_1<t_2$ such that $F_{\tau}(t_1)=0$ and $0<F_{\tau}(t_2)<1$. Let $B_1,B_2 \in \mathcal{B}(\mathbb{R})$, the formula of total probability and the fact that $\mathbb{P}(\tau\leq t_1)=0$ yields,
	\begin{align*}
		\mathbb{P}\left((\zeta_{t_1},\zeta_{t_2})\in B_1 \times B_2\vert \tau=r\right)&=\dint_{\mathbb{R}}\mathbb{P}\left((\zeta_{t_1},\zeta_{t_2})\in B_1 \times B_2\vert \tau=r,Z=z\right)\mathbb{P}_{Z}(\mathrm{d}z)\\
		&=\dint_{\mathbb{R}}\mathbb{P}\left((X_{t_1}^{r,z},X_{t_2}^{r,z})\in B_1 \times B_2\right)\mathbb{P}_{Z}(\mathrm{d}z)
		\\  &=\dint_{B_{1}\times B_{2}}q_{t_{1},t_{2}}(r,x_{1},x_{2})\,\mathrm{d}x_1\otimes(\mathrm{d}x_2+\mathbb{P}_{Z}^s(\mathrm{d}x_2)),
	\end{align*}
	where
	\begin{multline*}
		q_{t_1,t_2}(r,x_1,x_2)=\Big((1-a_{ac})\mathbb{I}_{\mathscr{Z}}(x_2)+a_{ac}f_Z(x_2)\mathbb{I}_{\mathscr{Z}^c}(x_2)\Big)\dfrac{f^X_{r-t_1}(x_2-x_1)f^X_{t_1}(x_1)}{f^X_{r}(x_2)}\mathbb{I}_{\{t_1<r\leq t_2\}}\\+f^X_{t_1}(x_1)f^X_{t_2-t_1}(x_2-x_1)\displaystyle\int_{\mathbb{R}} \dfrac{f^X_{r-t_2}(z-x_2)}{f^X_r(z)}\mathbb{P}_{Z}(\mathrm{d}z)\mathbb{I}_{\mathscr{Z}}(x_2)\,\mathbb{I}_{\{t_2<r\}}.
	\end{multline*} 
	Hence, it follows from Bayes' theorem that for all bounded measurable functions $h$ defined on $(0,+\infty)$, we have, $\mathbb{P}$-a.s.,
	\begin{align}
		&\mathbb{E}[h(\tau)|\zeta_{t_1}=x_1,\zeta_{t_2}=x_2]=\mathbb{I}_{\mathscr{Z}}(x_2)\dfrac{\dint_{t_1}^{t_2}h(r)\dfrac{f^X_{r-t_1}(x_2-x_1)}{f^X_{r}(x_2)}\mathbb{P}_{\tau}(\mathrm{d}r)}{\dint_{t_1}^{t_2}\dfrac{f^X_{r-t_1}(x_2-x_1)}{f^X_{r}(x_2)}\mathbb{P}_{\tau}(\mathrm{d}r)}+\mathbb{I}_{\mathscr{Z}^c}(x_2)\times\nonumber\\
		&\dfrac{a_{ac}f_Z(x_2)\dint_{t_1}^{t_2}h(r)\dfrac{f^X_{r-t_1}(x_2-x_1)}{f^X_{r}(x_2)}\mathbb{P}_{\tau}(\mathrm{d}r)+f^X_{t_2-t_1}(x_2-x_1)\dint_{t_2}^{+\infty}h(r)\displaystyle\int_{\mathbb{R}} \dfrac{f^X_{r-t_2}(z-x_2)}{f^X_r(z)}\mathbb{P}_Z(\mathrm{d}z)\mathbb{P}_{\tau}(\mathrm{d}r)}{a_{ac}f_Z(x_2)\dint_{t_1}^{t_2}\dfrac{f^X_{r-t_1}(x_{2}-x_{1})}{f^X_{r}(x_{2})}\mathbb{P}_{\tau}(\mathrm{d}r)+f^X_{t_2-t_1}(x_2-x_{1})\dint_{t_2}^{+\infty}\displaystyle\int_{\mathbb{R}} \dfrac{f^X_{r-t_2}(z-x_2)}{f^X_r(z)}\mathbb{P}_Z(\mathrm{d}z)\mathbb{P}_{\tau}(\mathrm{d}r)}.\label{eqtaugivenzeta1}
	\end{align}
	On the other hand, we have
	\begin{align}
		\mathbb{E}[g(Z)\vert\zeta_{t_1}=x_1,\zeta_{t_2}=x_2]&=\displaystyle\int_{0}^{+\infty}\mathbb{E}[g(Z)\vert\zeta_{t_1}=x_1,\zeta_{t_2}=x_2,\tau=r]\mathbb{P}_{\tau\vert \zeta_{t_1}=x_1,\zeta_{t_2}=x_2}(\mathrm{d}r)\nonumber\\
		&=\displaystyle\int_{0}^{+\infty}\mathbb{E}[g(Z)\vert\zeta^{r,Z}_{t_2}=x]\mathbb{P}_{\tau\vert \zeta_{t_1}=x_1,\zeta_{t_2}=x_2}(\mathrm{d}r).\label{eqV(tau,Z1)}
	\end{align}
	Thus, \eqref{eqV(r,Z)}, \eqref{eqtaugivenzeta1}, and \eqref{eqV(tau,Z1)} show that
	\begin{align}
		\mathbb{E}[g(Z)&\vert\zeta_{t_1},\zeta_{t_2}]=g(\zeta_{t_2})\mathbb{I}_{\mathscr{Z}}(\zeta_{t_2})+\nonumber\\
		&\dfrac{a_{ac}f_Z(\zeta_{t_2})g(\zeta_{t_2})+U_{t_1,t_2}(\zeta_{t_1},\zeta_{t_2})\dint_{t_2}^{+\infty}\displaystyle\int_{\mathbb{R}} g(z) \dfrac{f^X_{r-t_2}(z-\zeta_{t_2})}{f^X_r(z)}f_Z(z)\mathrm{d}z\mathbb{P}_{\tau}(\mathrm{d}r)}{a_{ac}f_Z(\zeta_{t_2})+U_{t_1,t_2}(\zeta_{t_1},\zeta_{t_2})\dint_{t_2}^{+\infty}\displaystyle\int_{\mathbb{R}} \dfrac{f^X_{r-t_2}(z-\zeta_{t_2})}{f^X_r(z)}f_Z(z)\mathrm{d}z\mathbb{P}_{\tau}(\mathrm{d}r)}\mathbb{I}_{\mathscr{Z}^c}(\zeta_{t_2}),\label{eqZgivenzete1zeta2}
	\end{align}
	where, 
	$$ U_{t_1,t_2}(x_{1},x_{2})=\dfrac{f^X_{t_2-t_1}(x_{2}-x_{1})}{\dint_{t_1}^{t_2}\dfrac{f^X_{r-t_1}(x_{2}-x_{1})}{f^X_{r}(x_{2})}\mathbb{P}_{\tau}(\mathrm{d}r)},\,\,x_1,\,x_2\in \mathbb{R},\,\, 0<t_1<t_2. $$
	From \eqref{erZgivenzeta}, we have, $\mathbb{P}$-a.s.,
	\begin{align}
		\mathbb{E}[g(Z)\vert\zeta_{t_2}]&=g(\zeta_{t_2})\mathbb{I}_{\mathscr{Z}}(\zeta_{t_2})+\nonumber\\
		&\dfrac{a_{ac}f_Z(\zeta_{t_2})g(\zeta_{t_2})+\dfrac{f^X_{t_2}(\zeta_{t_2})}{F_{\tau}(t_2)}\dint_{t_2}^{+\infty}\dint_{\mathbb{R}}g(z) \dfrac{f^X_{r-t_2}(z-\zeta_{t_2})}{f^X_r(z)}f_Z(z)\mathrm{d}z\mathbb{P}_{\tau}(\mathrm{d}r)}{a_{ac}f_Z(\zeta_{t_2})F_{\tau}(t_2)+\dfrac{f^X_{t_2}(\zeta_{t_2})}{F_{\tau}(t_2)}\dint_{t_2}^{+\infty}\dint_{\mathbb{R}} \dfrac{f^X_{r-_2}(z-\zeta_{t_2})}{f^X_r(z)}f_Z(z)\mathrm{d}z\mathbb{P}_{\tau}(\mathrm{d}r)}\mathbb{I}_{\mathscr{Z}^c}(\zeta_{t_2}).\label{eqZgivenzeta2}
	\end{align}
	Thus, it follows from \eqref{eqAssumeMarkov}, \eqref{eqZgivenzete1zeta2}, and \eqref{eqZgivenzeta2} that
	\begin{equation}
		U_{t_1,t_2}(\zeta_{t_1},\zeta_{t_2})=\dfrac{f^X_{t_2}(\zeta_{t_2})}{F_{\tau}(t_2)},
	\end{equation}
	contradiction, since for instance if $\tau$ is a two-point random variable such that $\mathbb{P}(\tau=T_1)=\mathbb{P}(\tau=T_2)=\frac{1}{2}$ and $t_1<T_1<t_2<T_2$, we have
	\begin{equation}
		\dfrac{f^X_{t_2-t_1}(\zeta_{t_2}-\zeta_{t_1})}{f^X_{T_1-t_1}(\zeta_{t_2}-\zeta_{t_1})}\,f^X_{T_1}(\zeta_{t_2})\neq f^X_{t_2}(\zeta_{t_2}).
	\end{equation}
	This completes the proof.
\end{proof}
Since the pinning point has a significant influence on the Markov property of the bridge process, it is worth noting that the Markov property still holds when considering the two-dimensional process consisting of both the pinning point and the bridge process.
\begin{theorem}
	The two-dimensional process $Y$ defined by $$Y_t=(Z,\zeta_{t}),\quad t\geq 0,$$
	is a Markov process with respect to its natural filtration. Moreover, for any $0<t<u$ and for every bounded measurable function $G$ defined on $\mathbb{R}^2$, we have
	\begin{align}
		\mathbb{E}[G(Y_u)\vert Y_{t}]&=G(Z,Z) \Big( \mathbb{I}_{\{\zeta_t=Z\}} 	+ \dfrac{\displaystyle\int_{t}^{u}\dfrac{f^X_{r-t}(Z-\zeta_t)}{f^X_r(Z)}\mathbb{P}_{\tau}(\mathrm{d}r)}{\displaystyle\int_{t}^{+\infty}\dfrac{f^X_{r-t}(Z-\zeta_t)}{f^X_r(Z)}\mathbb{P}_{\tau}(\mathrm{d}r)}\mathbb{I}_{\{\zeta_t\neq Z\}} \Big)\nonumber  \\ &+\dint_{\mathbb{R}} G(Z,y)f^X_{u-t}(y-\zeta_t)\dfrac{\dint_u^{+\infty}\,\dfrac{f^X_{r-u}(Z-y)}{f^X_{r}(Z)}\,\mathbb{P}_{\tau}(\mathrm{d}r)}{\displaystyle\int_{t}^{+\infty}\dfrac{f^X_{r-t}(Z-\zeta_t)}{f^X_r(Z)}\mathbb{P}_{\tau}(\mathrm{d}r)}\mathrm{d}y\,\mathbb{I}_{\{\zeta_t\neq Z\}}.
	\end{align}
\end{theorem}
\begin{proof}
	Using the fact that $\zeta^z$ is a Markov process with respect to its natural filtration (see \cite[Theorem 3.8]{EHL(Levy)}), for all $n\in \mathbb{N}$ and
	$0 < t_1 < t_2 <\ldots < t_n < u$, for all measurable functions $G$ defined on $\mathbb{R}^2$ such that $G(Z,\zeta_u)$ is integrable, we have, $\mathbb{P}$-a.s,
	\begin{align}
		\mathbb{E}[G(Y_u)\vert Y_{t_1}=y_1,\ldots,Y_{t_n}=y_n]&=\mathbb{E}[G(Z,\zeta_u)\vert \zeta_{t_1}=x_1,\ldots,\zeta_{t_n}=x_n,Z=z]\nonumber\\
		&=\mathbb{E}[G(z,\zeta_u^z)\vert \zeta_{t_1}^z=x_1,\ldots,\zeta_{t_n}^z=x_n]\nonumber\\
		&=\mathbb{E}[G(z,\zeta_u^z)\vert \zeta_{t_n}^z=x_n]\nonumber\\
		&=\mathbb{E}[G(Y_u)\vert Y_{t_n}=y_n].
	\end{align}
	Again, it follows from Theorem 3.8 in \cite{EHL(Levy)} that for $y=(z,x)\in\mathbb{R}^2$, for any $0<t<u$ and for every bounded measurable function $G$, we have
	\begin{align}
		\mathbb{E}[G(Y_u)\vert Y_{t}=y]&=[G(z,\zeta_u^z)\vert \zeta_{t}^z=x]=G(z,z) \Big( \mathbb{I}_{\{x=z\}} 	+ \dfrac{\displaystyle\int_{t}^{u}\dfrac{f^X_{r-t}(z-x)}{f^X_r(z)}\mathbb{P}_{\tau}(\hbox{d}r)}{\displaystyle\int_{t}^{+\infty}\dfrac{f^X_{r-t}(z-x)}{f^X_r(z)}\mathbb{P}_{\tau}(\hbox{d}r)}\mathbb{I}_{\{x\neq z\}} \Big)\nonumber  \\ &+\dint_{\mathbb{R}} G(z,y)f^X_{u-t}(y-x)\dfrac{\dint_u^{+\infty}\,\dfrac{f^X_{r-u}(z-y)}{f^X_{r}(z)}\,\mathbb{P}_{\tau}(\hbox{d}r)}{\displaystyle\int_{t}^{+\infty}\dfrac{f^X_{r-t}(z-x)}{f^X_r(z)}\mathbb{P}_{\tau}(\hbox{d}r)}\hbox{d}y\,\mathbb{I}_{\{x\neq z\}}.
	\end{align}
	This completes the proof.
\end{proof}

\end{document}